\documentclass[psamsfonts,12pt]{amsart}

\usepackage[margin=1in]{geometry}  % set the margins to 1in on all sides
\usepackage{graphicx}              % to include figures
\usepackage{amsmath}               % great math stuff
\usepackage{amsfonts}              % for blackboard bold, etc
\usepackage{amsthm}                % better theorem environments
\usepackage{verbatim}              % to comment out large swathes of text with the comment environment
\usepackage{indentfirst}           % to indent the first paragraph of a section
\usepackage{tikz-cd}

\newtheorem{thm}{Theorem}[section]
\newtheorem{lem}[thm]{Lemma}

\newtheorem{cor}[thm]{Corollary}

\theoremstyle{remark}
       \newtheorem*{rmk}{Remark}
\theoremstyle{remark}
       \newtheorem*{example}{Example}

\newcommand{\ZZ}{\mathbb{Z}}      % for Integers
\newcommand{\Fq}{\mathbb{F}_q}
\newcommand{\Pn}{\mathbb{P}^n} % for projective n-space over F_q
\newcommand{\PFq}{\mathbb{P}^n_{\mathbb{F}_q}}
\newcommand{\Iz}{\mathcal{I}_Z}

\title[Random hypersurfaces and embedding curves]{Random hypersurfaces and embedding curves in surfaces over finite fields}

\author{Joseph Gunther}

\address{Department of Mathematics, The Graduate Center, City University of New York (CUNY); 365 Fifth Avenue,
New York, NY 10016, U.S.A. }
\email{JGunther@gradcenter.cuny.edu}

\begin{document}

\begin{abstract}
We use Poonen's closed point sieve to prove two independent results. First, we show that the obvious obstruction to embedding a curve in an unspecified smooth surface is the only obstruction over a perfect field, by proving the finite field analogue of a Bertini-type result of Altman and Kleiman.  Second, we prove a conjecture of Vakil and Wood on the asymptotic probability of hypersurface sections having a prescribed number of singularities.
\end{abstract}

\maketitle

\section{Introduction}

Poonen's geometric closed point sieve was first introduced in \cite{poonen04} to prove a finite field version of the classical Bertini smoothness theorem.  The sieve has since been applied and adapted to a range of subjects, including point-counting distributions within families of curves \cite{planecurves, bucurkedlaya, semiample} and arithmetic dynamics \cite{poonendynamics}.  In this paper, we use it to prove embedding results for quasi-projective schemes over finite fields, as well as to prove a hypersurface stabilization conjecture of Vakil and Wood. 

Given a curve, when does there exist some smooth surface into which it can be embedded? There is an obvious requirement: the curve must have no more than two tangent directions at any point, since this would be true on an ambient smooth surface.  Altman and Kleiman proved that over an infinite perfect field, this local obstruction is the only obstruction \cite{KA}.  In this paper we prove the same for finite fields, thus removing their infinite hypothesis.  The result follows from Corollary \ref{curveembedding} of the following theorem.  (Each $\zeta$ below indicates a zeta function, and for ease of notation we define the empty set to have dimension $-\infty$; see Section 2 for full notation and definitions.)

\begin{thm} \label{embedding}
Let $X$ be a smooth subscheme of $\PFq$ of dimension $m$, $Z$ a closed subscheme of $\PFq$, and $\mathcal{H}_{Z,d}$ the set of degree $d$ hypersurfaces in $\PFq$ that contain $Z$.  Let $V=X \cap Z$, and for any $e \geq 0$, let $V_e$ be the (locally closed) subset of $V$ whose closed points are exactly those of local embedding dimension $e$ in $V$.  Then if ${\displaystyle \max_e} \{dim (V_e) + e\} < m$, we have $$\lim_{d \to \infty} \frac{\#\{ H \in \mathcal{H}_{Z,d} \ | \ X \cap H \text{ is smooth of dimension } m-1\}}{\#\mathcal{H}_{Z,d}}=\frac{1}{\zeta_{X-V}(m+1){\displaystyle \prod_{e} \zeta_{V_e}(m-e)}}.$$
Conversely, if for some value of $e$ we have $dim (V_e) + e \geq m$, then the limit is 0.
\end{thm}

Roughly speaking, the above theorem says that, with positive probability, a hypersurface section of a smooth scheme $X$ containing a given subscheme $V$ is again smooth, provided that the dimension and singularities of the subscheme are adequately controlled. Furthermore, that positive probability is given by special values of zeta functions, which is what a naive point-by-point heuristic predicts.

\begin{rmk}In the case where the subscheme $V$ is smooth, Theorem \ref{embedding} gives the central theorem of \cite{poonen08}.  While our result is more general, its proof is ultimately inspired by that paper.  Over an infinite perfect field, under similar hypotheses on local embedding dimensions, the existence of smooth hypersurface sections was proved in \cite[Theorem 7]{KA}.  After submitting this paper for publication, the author learned that Theorem \ref{embedding} was independently obtained by Wutz in her recent thesis \cite[Theorem 2.1]{wutz}.
\end{rmk}

\begin{cor} \label{curveembedding}
Let $C$ be a reduced quasi-projective curve over $\Fq$, not necessarily smooth, irreducible, or projective.  Then there exists a smooth $r$-dimensional scheme over $\Fq$ in which $C$ can be embedded if and only if the maximal ideal at each closed point of $C$ can be generated by $r$ elements.  If $C$ is projective, the smooth scheme can be chosen projective as well.
\end{cor}

\begin{proof}[Proof of Corollary \ref{curveembedding} from Theorem \ref{embedding}]
Necessity is clear.  For sufficiency, consider $C$ embedded in $\PFq$ for some $n$. If $n = r$, we're done.  If $n < r$, embed $\PFq$ linearly into $\mathbb{P}^r_{\mathbb{F}_q}$, and we're again done.  Otherwise, let $Z=\bar{C}$ and $X=\PFq - (\bar{C} - C)$. Applying Theorem \ref{embedding} recursively $n-r$ times to find smooth hypersurface sections containing $X \cap Z$, we construct a smooth, $r$-dimensional $\Fq$-scheme $X \cap H_1 \cap \ldots \cap H_{n-r}$ containing $C$.  It is projective if $C$ is.
\end{proof}

\begin{rmk}
Over an infinite perfect field, this corollary was proven in \cite[Corollary 9]{KA}, using methods inspired by Bloch's thesis \cite[Proposition 1.2]{bloch}.  This paper shows the corollary is in fact true over any perfect field.  The starting idea of both proofs is the same: embed your curve in some large projective space, and then try to show there exist hypersurfaces that contain your curve and whose mutual intersection is smooth of the correct dimension.  Altman and Kleiman's proof in the infinite case proceeds via a Bertini-type argument that fails over finite fields since $\Fq$-points aren't dense in a rational variety; instead, we adapt Poonen's closed point sieve to prove the quantitative result in Theorem \ref{embedding}.\end{rmk}

The local embedding dimension at a simple node or cusp on a reduced curve is 2, so we have the following special case.
\begin{cor} \label{curveinsurface}
Let $C$ be a reduced, quasi-projective curve over $\Fq$ with only simple nodes and cusps.  Then $C$ can be embedded in some smooth surface over $\Fq$.
\end{cor}

\begin{rmk}
In his thesis \cite[Theorem 1.0.2]{nguyen}, N. Nguyen proved a different embedding result, answering the question of when a smooth variety $X$ over $\Fq$ of dimension $m$ admits a closed immersion into $\PFq$, for $n \geq 2m+1$.  In that case, the only obstruction is also an obvious one, though of an arithmetic nature: embedding fails exactly if, for some $e \geq 1$, $X$ has more closed points of degree $e$ than $\PFq$ itself.
\end{rmk}

Theorem \ref{embedding} also applies to higher-dimensional schemes, not just curves.  In particular, we obtain some appealing probabilistic corollaries about subschemes $V \subset \PFq$ if we take $X=\PFq$ and $Z=\bar{V}$ in the theorem.

\begin{cor} \label{randomhyp}
Let $V \subset \PFq$ be an arbitrary subscheme.  Then the probability that a random hypersurface containing $V$ will be smooth is $$\begin{cases}
1/[\zeta_{\mathbb{P}^n-\bar{V}}(n+1){\displaystyle \prod_{e} \zeta_{(\bar{V})_e}(n-e)}], & \text{if }{\displaystyle \max_e} \{dim ((\bar{V})_e) + e\} < n, \\
0, & \text{otherwise}.
\end{cases}$$
\end{cor}

\begin{rmk}
By rationality of the zeta function \cite{dwork}, the probabilities in Theorem \ref{embedding} and Corollary \ref{randomhyp} are always rational numbers.
\end{rmk}

\begin{example}
Let $C$ be the rational curve defined in $\mathbb{P}^3_{\Fq}$ by $w=0$ and $y^2z-x^3+x^2z=0$.  Then $\zeta_{V_1}(s)^{-1}=\frac{1-q^{1-s}}{1-q^{-s}}$, $\zeta_{V_2}(s)^{-1}=1-q^{-s}$, and $\zeta_{X-V}(s)^{-1}=(1-q^{-s})(1-q^{2-s})(1-q^{3-s})$.  So, for example, the probability that a hypersurface in $\mathbb{P}^3_{\mathbb{F}_2}$ containing $C$ will be smooth is $[\zeta_{X-V}(4) \cdot \zeta_{V_1}(2) \cdot \zeta_{V_2}(1)]^{-1} = \frac{15}{128}$.
\end{example}

\begin{rmk}
We should caution that just because an asymptotic probability in Theorem \ref{embedding} or Corollary \ref{randomhyp} is 0, this does not in general rule out the existence of {\em any} smooth hypersurface sections containing the given scheme.  For example, the non-reduced scheme cut out by $y^2=0$ and $z=0$ in $\mathbb{A}^3_{\Fq} \subset \mathbb{P}^3_{\Fq}$ is contained in smooth affine hypersurfaces of arbitrarily high degree (such as those given by $z-y^d=0$); however, in accordance with Theorem \ref{embedding}, the proportion of smooth hypersurfaces decreases to 0 (exponentially with the degree, in fact).  Conversely, a curve in $\mathbb{P}^3_{\Fq}$ with a point of local embedding dimension 3 is contained in no smooth hypersurfaces at all.
\end{rmk}

The second main theorem of this paper is also an application of Poonen's sieve; in Section 4, we prove a recent conjecture of Vakil and Wood on hypersurface sections with a prescribed number of singularities.  Before stating it, we provide some motivation.

Let $X$ be a smooth, quasi-projective, $m$-dimensional scheme over $\Fq$.  Roughly speaking, \cite[Theorem~1.1]{poonen04} showed that a hypersurface section of $X$ has zero singularities with probability $\frac{1}{\zeta_X(m + 1)}$.  At the other extreme, \cite[Theorem~3.2]{poonen04} showed that a section has infinitely many singularities with probability 0.  It is then natural to ask how the probabilities are distributed across the remaining possible numbers of singularities (one, two, etc.): $$\frac{1}{\zeta_X(m + 1)} \ + \ ? \ + \ ? \ + \ldots = 1.$$

To answer this question, we need a little notation.  Let $X$ be a finite-type scheme over $\Fq$, and define $Z_X(t)=\sum_{n=0}^\infty |(Sym^nX)(\Fq)|t^n$.  Then a standard computation shows that $Z_X(q^{-s}) = \zeta_X(s)$, as defined in the next section.  The points of $Sym^nX$ correspond to formal sums of $n$ points on $X$, with possible repetition; let $Sym^n_{[\ell]} X$ be the natural subset comprising just those sums supported on exactly $\ell$ geometric points.  Analogously, define $Z_X^{[\ell]}(t)= \sum_{n=0}^\infty |(Sym^n_{[\ell]}X)(\Fq)|t^n$, and let $\zeta_X^{[\ell]}(s)=Z_X^{[\ell]}(q^{-s}).$  Based on their own motivic results about the Grothendieck ring of varieties, Vakil and Wood conjectured the following generalization of Poonen's Bertini theorem \cite[Conjecture~A]{vakwood}.

\begin{thm} \label{conjecture}
Let $X$ be a smooth $m$-dimensional subscheme of $\PFq$, $\ell \geq 0$ an integer, and $\mathcal{H}_d$ the set of degree $d$ hypersurfaces in $\PFq$.  Then $$\lim_{d \to \infty} \frac{\#\{ H \in \mathcal{H}_d \ | \ X \cap H \text{ has exactly } \ell \text{ singular geometric points}\}}{\#\mathcal{H}_d}=\frac{\zeta_X^{[\ell]}(m+1)}{\zeta_X(m+1)}.$$
\end{thm}

\begin{rmk}
This gives the distribution of probabilities over all possible numbers of singularities, in terms of a natural decomposition of the zeta function:  $$\frac{1}{\zeta_X(m + 1)} \ + \ \frac{\zeta_X^{[1]}(m+1)}{\zeta_X(m+1)} \ + \ \frac{\zeta_X^{[2]}(m+1)}{\zeta_X(m+1)} \ + \ldots = 1.$$
\end{rmk}

\begin{example}
What is the probability that a plane curve is singular at exactly one geometric point?  For $X = \mathbb{P}^2_{\Fq}$, we have $\zeta_X^{[1]}(s)=\frac{q^2+q+1}{q^s-1}$, and so the probability is $\frac{\zeta_X^{[1]}(3)}{\zeta_X(3)}=\frac{(q^3-1)(q^2-1)}{q^6}$.  For $\mathbb{F}_2$, this probability is $\frac{21}{64}$.  Coincidentally, by \cite[Section~3.5]{poonen04}, this is the same as the probability that it's smooth; thus over $\mathbb{F}_2$, a plane curve is precisely as likely to be smooth as it is to have exactly one singularity.  Over any other finite field, a random plane curve is more likely to be smooth than singular.
\end{example}

\noindent \textbf{Acknowledgments.} I thank Johan de Jong, Raymond Hoobler, and Joe Kramer-Miller for helpful conversations.

\section{Notation and Conventions}

Let $X$ be a scheme of finite type over $\ZZ$.  The zeta function of $X$ is defined as $$\zeta_X(s)=\prod_{\text{closed points }P \in X}\frac{1}{1-|\kappa(P)|^{-s}},$$ where $\kappa(P)$ is the residue field of $P$.  The product converges for $\text{Re}(s) > \text{dim }X$ (\cite[Theorem~1]{serrezeta}).  In the particular case where $X$ is a scheme of finite type over $\Fq$, we have that $$\zeta_X(s)=\prod_{\text{closed }P \in X}\frac{1}{1-q^{-s \text{ deg }P}}=\text{exp}\left(\sum_{n=1}^\infty \frac{|X(\mathbb{F}_{q^n})|}{n}q^{-ns}\right).$$

Following \cite{poonen04} and \cite{poonen08}, we wish to measure the density of sets of homogeneous $\Fq$-polynomials, within both the space of all such polynomials and just those vanishing on a given subscheme of $\Pn_{\Fq}$. We'll often speak informally of these densities as probabilities.  Let $S=\Fq[x_0,x_1,\ldots,x_n]$, let $S_d$ be its degree $d$ homogeneous part, and let $S_{homog} = \bigcup_{d \geq 0} S_d$.  For any $\mathcal{P} \subset S_{homog}$, we define the \emph{density} of $\mathcal{P}$ to be $$\mu(\mathcal{P}) = \lim_{d \to \infty} \frac{\#\mathcal{P} \cap S_d}{\#S_d}$$
if the limit exists.

To define the density relative to a closed subscheme $Z$ of $\Pn_{\Fq}$, let $I_{homog}$ denote the homogeneous elements of $S$ that vanish on $Z$, and $I_d$ the degree $d$ part.  For $\mathcal{P} \subset I_{homog}$, we define its \emph{density relative to $Z$} as $$\mu_Z(\mathcal{P}) = \lim_{d \to \infty} \frac{\#\mathcal{P} \cap I_d}{\#I_d}$$ if the limit exists.

Note that Theorem \ref{embedding} is equivalent to a statement about $\mu_Z$; we'll use this notation in its proof.  Theorem \ref{conjecture} is technically a statement about $\mu$, but we will simply speak of probabilities in its proof.  For $f \in S_d$, let $H_f= \text{Proj}(S/(f))$ be the associated hypersurface.  All intersections and closures are scheme-theoretic, and a subscheme means a closed subscheme of an open subscheme.  We use the convention that a product over an empty set is 1, and that the dimension of the empty set is $-\infty$.

Following \cite[Section II.7]{hartshorne}, for a morphism $Y\rightarrow X$ and a sheaf of ideals $\mathcal{I}$ on $X$, we write $\mathcal{I} \cdot \mathcal{O}_Y$ for the {\em inverse image ideal sheaf} in $\mathcal{O}_Y$.  For the definition of a simple singularity on a curve (also known as an ADE-singularity), we refer the reader to \cite{simple}.

\section{Embedding Dimension Theorem}
\label{sect:embedding}

Let $X$ and $Z$ be as in Theorem \ref{embedding}, with $I \subset S$ the vanishing ideal of $Z$.  We define the \emph{local embedding dimension} $e(P)$ of a closed point $P$ of a scheme to be the minimal number of generators for the maximal ideal $\mathfrak{m}_P$ in its stalk, or equivalently by Nakayama's Lemma, the dimension of $\mathfrak{m}_P/\mathfrak{m}_P^2$ over the residue field $\kappa(P)$.  In this section, $\Pn = \Pn_{\Fq}$, and the local embedding dimension of a point $P$ will always mean as a point of $V=X\cap Z$.  For ease of comparison, we parallel the structure of \cite{poonen08}.

\subsection{Singular Points of Low Degree}

Fix any $c$ such that $S_1 I_d = I_{d+1}$ for all $d \geq c$; for example, choose a finite homogeneous generating set for the ideal, and let $c$ be the maximal degree of its elements.  The following interpolation lemma is \cite[Lemma~2.1]{poonen08}.

\begin{lem} \label{explicitvanishing}
Let $Y$ be a finite subscheme of $\Pn$. Then the restriction map $$\phi_d: I_d = H^0(\Pn,\mathcal{I}_Z(d)) \rightarrow H^0(Y, \mathcal{I}_Z \cdot \mathcal{O}_Y(d))$$ is surjective for $d \geq c + h^0(Y, \mathcal{O}_Y)$.
\end{lem}

\begin{lem} \label{lowprob}
Suppose $\mathfrak{m} \subset \mathcal{O}_X$ is the ideal sheaf of a closed point $P \in X$.  Let $Y \subset X$ be the closed subscheme whose ideal sheaf is $\mathfrak{m}^2 \subset \mathcal{O}_X$.  Then for any $d \in \ZZ_{\geq 0}$, $$\#H^0(Y, \Iz \cdot \mathcal{O}_Y(d)) = \begin{cases}
q^{(m-e(P))\text{deg }P}, & \text{if }P \in V, \\
q^{(m+1)\text{deg }P}, & \text{if }P \not \in V.
\end{cases}$$
\end{lem}

\begin{proof}
Because $X$ is smooth, the space $H^0(Y,\mathcal{O}_Y(d))$ has a two-step filtration whose quotients have dimensions $1$ and $m$ over the residue field $\kappa(P)$.  Thus $\#H^0(Y,\mathcal{O}_Y(d)) = q^{(m+1)\text{deg }P}$.  If $P \in V=X \cap Z$, then $H^0(Y, \mathcal{O}_{Z \cap Y}(d))$ has a filtration whose quotients have dimensions $1$ and $e(P)$ over $\kappa(P)$; if $P \not \in V$, then $H^0(Y,\mathcal{O}_{Z \cap Y}(d))=0$.  Taking global sections for the exact sequence $$0 \rightarrow \Iz \cdot \mathcal{O}_Y(d) \rightarrow \mathcal{O}_Y(d) \rightarrow \mathcal{O}_{Z \cap Y}(d) \rightarrow 0$$ (taking global sections is exact on a zero-dimensional Noetherian scheme) gives

\begin{align*}
\#H^0(Y,\Iz \cdot \mathcal{O}_Y(d)) &= \frac{\#H^0(Y,\mathcal{O}_Y(d))}{\#H^0(Y,\mathcal{O}_{Z \cap Y}(d))}\\
&=\begin{cases}
q^{(m+1)\text{deg }P}/q^{(e(P)+1)\text{deg }P}, & \text{if }P \in V \\
q^{(m+1)\text{deg }P}, & \text{if }P \not \in V.
\end{cases}
\end{align*}
\end{proof}

For $S$ a scheme of finite type over $\Fq$, let $S_{< r}$ be the set of closed points of $S$ of degree less than $r$.  Define $S_{> r}$ and $S_{\geq r}$ similarly.

\begin{lem}[Singularities of low degree] \label{lowdensity}
Let notation and hypotheses be as in Theorem \ref{embedding}, and define $$\mathcal{P}_r = \{f \in I_{\text{homog}} \ | \ X \cap H_f \text{ is smooth of dimension } m-1 \text{ at all } P \in X_{< r}\}.$$
Then $$\mu_Z(\mathcal{P}_r) = \left(\prod_{P \in (X - V)_{< r}} (1-q^{-(m+1)\text{deg }P}) \right) \cdot \prod_e \prod_{P \in (V_e)_{< r}} (1-q^{-(m-e)\text{deg }P}).$$
\end{lem}

\begin{proof}
Let $X_{< r} = \{P_1, \ldots, P_k\}$.  Let $\mathfrak{m}_i$ be the ideal sheaf of $P_i$ on $X$.  Let $Y_i$ be the closed subscheme of $X$ with ideal sheaf $\mathfrak{m}_i^2 \subset \mathcal{O}_X$, and let $Y=\bigcup Y_i$.  Then $H_f \cap X$ is not smooth of dimension $m - 1$ at $P_i$ exactly if the restriction of $f$ to a section of $\mathcal{O}_{Y_i}(d)$ is zero.

By Lemma \ref{explicitvanishing}, the restriction map $\phi_d: I_d \rightarrow H^0(Y, \mathcal{I}_Z \cdot \mathcal{O}_Y(d))$ is surjective for $d >> 0$, and as this is a linear map, its values are equidistributed.  So $\mu_Z(\mathcal{P}_r)$ just equals the fraction of elements in $H^0(\Pn, \Iz \cdot \mathcal{O}_Y(d))$ which are nonzero when restricted to each $Y_i$, which is constant.  Thus, by Lemma \ref{lowprob}, 
\begin{align*}
\mu_Z(\mathcal{P}_r) &= \prod_{i=1}^s \frac{\#H^0(Y_i, \Iz \cdot \mathcal{O}_{Y_i}(d)) - 1}{\#H^0(Y_i, \Iz \cdot \mathcal{O}_{Y_i}(d))}\\
&= \left(\prod_{P \in (X - V)_{< r}} (1-q^{-(m+1)\text{deg }P}) \right) \cdot \prod_e \prod_{P \in (V_e)_{< r}} (1-q^{-(m-e)\text{deg }P}).
\end{align*}
\end{proof}

\begin{cor}
If $\text{dim}(V_e) + e < m$ for all $e$, then $$\lim_{r \rightarrow \infty} \mu_Z(\mathcal{P}_r) = \frac{1}{\zeta_{X-V}(m+1){\displaystyle \prod_{e} \zeta_{V_e}(m-e)}}.$$
\end{cor}

\begin{proof}
The products in Lemma \ref{lowdensity} are the reciprocals of the partial products in the definition of the zeta functions.  For convergence, we need $m - e > \text{dim}(V_e)$ for each $e$ (\cite[Corollary~5]{langweil}), which is our hypothesis exactly.
\end{proof}

\begin{cor}
If $\text{dim}(V_e) + e \geq m$ for some $e$, then $\lim_{r \rightarrow \infty} \mu_Z(\mathcal{P}_r) = 0.$
\end{cor}

\begin{proof}
By \cite[Corollary~5]{langweil}, $\zeta_{V_e}(s)$ has a pole at $s= \text{dim}(V_e)$, so the product in Lemma \ref{lowdensity} converges to 0.  This proves the second part of Theorem \ref{embedding}.
\end{proof}

\subsection{Singular Points of Medium Degree}

\begin{lem} \label{mediumprob}
Let $P \in X$ be a closed point with deg $P \leq \frac{d-c}{m+1}$.  Then the fraction of $f \in I_d$ such that $X \cap H_f$ is not smooth of dimension $m - 1$ at $P$ equals $$\begin{cases}
q^{-(m-e(P))\text{deg }P}, & \text{if } P \in V\\
q^{-(m+1)\text{deg }P}, & \text{if } P \not \in V.
\end{cases}$$
\end{lem}

\begin{proof}
Let $Y$ be as in Lemma \ref{lowprob}.  Then $\#H^0(Y,\Iz \cdot \mathcal{O}_{Y_i}(d))$ is given by the same lemma, which serves to calculate the desired fraction by Lemma \ref{explicitvanishing}.
\end{proof}

Define the \emph{upper density} $\bar{\mu}_Z(\mathcal{P})$ as the lim sup of the expression used to define $\mu_Z$.

\begin{lem}[Singularities of medium degree] 
Define

\begin{align*}
\mathcal{Q}_r^{\text{medium}} = \bigcup_{d \geq 0} \{f \in I_d \ | &\text{there exists } P \in X \text{ with } r \leq \text{ deg } P \ \leq \frac{d-c}{m+1}\\
& \text{such that } X \cap H_f \text{ is not smooth of dimension } m-1 \text{ at } P\}.
\end{align*}

Then ${\displaystyle \lim_{r \rightarrow \infty}} \bar{\mu}_Z(\mathcal{Q}_r^{\text{medium}}) = 0$.
\end{lem}

\begin{proof}
By Lemma \ref{mediumprob}, we have

\begin{align*}
\frac{\#(\mathcal{Q}_r^{\text{medium}}\cap I_d)}{\#I_d} &\leq \sum_e \left(\sum_{\substack{P \in V_e\\
r \leq \text{ deg } P \ \leq \frac{d-c}{m+1}}} q^{-(m-e)\text{deg } P}\right) + \sum_{\substack{P \in X-V\\
r \leq \text{ deg } P \ \leq \frac{d-c}{m+1}}} q^{-(m+1)\text{deg } P}\\
&\leq \sum_e \left(\sum_{P \in (V_e)_{\geq r}} q^{-(m-e)\text{deg }P}\right) + \sum_{P \in (X - V)_{\geq r}} q^{-(m+1)\text{deg } P}.
\end{align*}

By \cite[Lemma~1]{langweil}, a $k$-dimensional variety has $O(q^{kl})$ closed points of degree $l$; applied to each $V_e$ and $X - V$, we see as in \cite[Lemma~3.2]{poonen08} that the above expression is $O(q^{-r})$ as $r \rightarrow \infty$, under our assumption that $\text{dim}(V_e) + e < m$ for each $e$.
\end{proof}

\subsection{Singular Points of High Degree}

\begin{lem}[Singularities of high degree off $V$] \label{highoff}
Define \\
$\mathcal{Q}_{X-V}^{\text{high}} = {\displaystyle \bigcup_{d \geq 0}}\{f \in I_d \ | \ \exists P \in (X - V)_{> \frac{d-c}{m+1}} \text{s.t. } X \cap H_f  \text{ isn't smooth of dimension } m-1 \text{ at } P \}.$ Then $\bar{\mu}_Z(\mathcal{Q}_{X-V}^{\text{high}})=0$.
\end{lem}

\begin{proof}
The proof of \cite[Lemma~4.2]{poonen08} works without change.
\end{proof}

\begin{lem}[Singularities of high degree on $V_e$]
For any $e$ such that $V_e$ is not empty, define $\mathcal{Q}_{V_e}^{\text{high}} = {\displaystyle \bigcup_{d \geq 0}}\{f \in I_d \ | \ \exists P \in (V_e)_{> \frac{d-c}{m+1}} \text{s.t. } X \cap H_f \text{ isn't smooth of dimension } m-1 \text{ at } P \}.$\\
Then $\bar{\mu}_Z(\mathcal{Q}_{V_e}^{\text{high}})=0.$
\end{lem}

\begin{proof}
As the union of finitely many density 0 sets will be density 0, it suffices to prove the lemma with $X$ replaced by each of the sets in an open covering of $X$, so we may assume $X$ is contained in $\mathbb{A}_{\Fq}^n = \{x_0 \neq 0 \} \subset \Pn$, and we may dehomogenize by setting $x_0=1$.  This identifies $I_d \subset S_d \subset \Fq[x_0, \ldots, x_n]$ with subspaces $I'_d \subset S'_d \subset A = \Fq[x_1,\ldots,x_n]$.

Since $V$ isn't assumed smooth, we can't take it to be locally cut out by a system of local parameters, as is done in \cite{poonen08}.  Instead, fix a closed point $v \in V_e$.  Recall the exact sequence of sheaves on $V$ \cite[Section II.8]{hartshorne}: $$\mathcal{I}_V/\mathcal{I}_V^2 \rightarrow \Omega_X^1 \otimes \mathcal{O}_V \rightarrow \Omega_V^1 \rightarrow 0.$$  Thus we can choose a system of local parameters $t_1,\ldots,t_n \in A$ at $v$ on $\mathbb{A}_{\Fq}^n$ such that $t_{m+1}=t_{m+2}= \ldots = t_n =0$ defines $X$ locally at $v$, while $t_1, \ldots, t_{m-e}$ vanish on $V$.  In fact, since $V = X \cap Z$, we may choose  $t_1, \ldots, t_{m-e}$ vanishing on $Z$.

Now $dt_1, \ldots, dt_n$ are an $\mathcal{O}_{\mathbb{A}_{\Fq}^n,v}$-basis for the stalk $\Omega^1_{\mathbb{A}_{\Fq}^n,v}$.  Let $\partial_1,\ldots, \partial_n$ be the dual basis of the stalk $\mathcal{T}_{\mathbb{A}_{\Fq}^n,v}$ of the tangent sheaf.  Choose $s \in A$ with $s(v) \neq 0$ to clear denominators so that $D_i = s \partial_i$ gives a global derivation $A \rightarrow A$ for $i=1,\ldots,n$.  Then there is a neighborhood $U$ of $v$ in $\mathbb{A}_{\Fq}^n$ such that $U \cap \{ t_{m+1}=t_{m+2} = \ldots = t_n =0\} = U \cap X$, $\Omega^1_U = \oplus_{i=1}^n \mathcal{O}_U dt_i$, and $s \in \mathcal{O}_U^*$.  For $f \in I'_d$, $H_f \cap X$ fails to be smooth of dimension $m - 1$ at a point $P \in V_e \cap U$ if and only if $f(P)=(D_1f)(P)= \ldots = (D_mf)(P) = 0$.

Let $N = \text{dim}(V_e)$, $\tau=\text{max}_i\{\text{deg } t_i\}$ and $\gamma = \lfloor{\frac{d-\tau}{p}} \rfloor$, where $p$ is the characteristic of $\Fq$.  Given choices of $f_0 \in I'_d$, and $g_i \in S'_\gamma$ for $i=1, \ldots N+1$, let $f=f_0 + g_1^pt_1 + \ldots + g_{N+1}^pt_{N+1}$.  By hypothesis, $N+1 = \text{dim}(V_e) + 1 \leq m-e$, so we have each $t_i \in I'_d$. Given all possible choices of $f_0, g_1, \ldots g_{N+1}$, $f$ realizes every element of $I'_d$ the same number of times, because of $f_0$ (i.e. $f$ is a random element of $I'_d$).

This has served to make the derivatives partially independent of each other: note that for $i \leq N+1$, $D_if=D_if_0+sg_i^p$.  Given choices of $f_0, g_1, \ldots, g_i$, let $W_i = V_e \cap \{D_1f=\ldots=D_if=0\}$, which depends only on these choices. As in \cite[Lemma~2.6]{poonen04}, for $1 \leq i \leq N$, the fraction of choices of $f_0, g_1, \ldots, g_i$ such that $\text{dim}(W_i) \leq N - i$ goes to 1 as $d\rightarrow\infty$.  In particular, for most choices, $W_N$ is finite.

\begin{comment}
By Bezout's theorem, the number of $(N-i)$-dimensional components of $W_i$ is $O(d^i)$.  Given such a component, there are either no choices of $g_{i+1}$ such that $D_{i+1}f$ vanishes on it, or all such choices lie in one coset of the subspace of $\Fq$-polynomials vanishing on $W_i$.  By standard Hilbert polynomial facts, this subspace has $O(q^{d^{n-i}})$ elements.  Thus the number of choices of $g_{i+1}$ such that $D_{i+1}f$ vanishes on some $(N-i)$-dimensional component of $W_i$ is $O(d^iq^{d^{n-i}})$, which as a proportion of all possible choices goes to 0.  Thus the claim is established.  In particular, the fraction of choices of $f_0, g_1, \ldots, g_N$ such that $W_N$ is finite goes to 1 as $d\rightarrow \infty$.
\end{comment}

Next, as in \cite[Lemma~4.3]{poonen08}, given \emph{any} choice of $f_0, g_1, \ldots, g_N$ such that $W_N$ is finite, the fraction of choices of $g_{N+1}$ such that $(V_e)_{> \frac{d-c}{m+1}} \cap W_{N+1} = \emptyset$ goes to 1 as $d \rightarrow \infty$.  In conclusion (the product of two quantities that both go to 1 itself goes to 1), $\bar{\mu}_Z(\mathcal{Q}_{V_e}^{\text{high}})=0$.
\end{proof}

\begin{proof}[Proof of Theorem \ref{embedding}]
Let $\mathcal{P}=\{f \ | \ X \cap H_f \text{ is smooth of dimension } m-1\}$.  Then we have $\mathcal{P} \subset \mathcal{P}_r \subset \mathcal{P} \cup \mathcal{Q}_r^{\text{medium}} \cup \mathcal{Q}_{X-V}^{\text{high}} \cup ({\displaystyle \cup_e \mathcal{Q}_{V_e}^{\text{high}}})$, so by the preceding results $$\mu_Z(\mathcal{P})={\displaystyle \lim_{r \rightarrow \infty} \mu_Z(\mathcal{P}_r)=\frac{1}{\zeta_{X-V}(m+1){\displaystyle \prod_{e} \zeta_{V_e}(m-e)}}}.$$
\end{proof}

\section{The Probability of a Hypersurface Section Having a Given Number of Singularities}

\begin{proof}[Proof of Theorem \ref{conjecture}]
Fix a value of $\ell \geq 1$.  Suppose we have $r$ distinct closed points $\{P_1, \ldots, P_r\}$ of $X$, of any degrees $\lambda_1, \ldots, \lambda_r$ such that $\sum \lambda_i = \ell$.  Then the contribution of zero-cycles supported on exactly this set to $Z_X^{[\ell]}(t)$ is $\prod_{i=1}^r \left(\sum_{n=1}^\infty t^{n\lambda_i} \right)=\prod_{i=1}^r \frac{t^{\lambda_i}}{1-t^{\lambda_i}}$.  Plugging in $q^{-(m+1)}$ gives that their contribution to $\zeta_X^{[\ell]}(m+1)$ is $\prod_{i=1}^r \frac{q^{-\lambda_i(m+1)}}{1-q^{-\lambda_i(m+1)}}$.

On the other hand, consider the probability that an $\Fq$-hypersurface section $X \cap H$ of $X$ is singular at exactly the points $\{P_1, \ldots, P_r\}$. (Note that since $X$ and $H$ are both defined over $\Fq$, $X \cap H$ is singular at a geometric point if and only if it's singular at all of the point's $\Fq$-conjugates.)  Let $\mathfrak{m}_i$ be the ideal sheaf of the point $P_i$, and let $Z_i$ be the subscheme of $X$ defined by $\mathfrak{m}_i^2$.  Let $Z = \bigcup Z_i$.  Then by Theorem 1.2 (Bertini with Taylor conditions) of \cite{poonen04} applied to $T=\{0\} \times \ldots \times \{0\}$, the probability that an $\Fq$-hypersurface section of $X$ is singular at exactly the points $\{P_1, \ldots, P_r\}$ is $$\frac{1}{q^{\sum_i \lambda_i (m+1)}} \cdot \frac{1}{\zeta_{X-Z}(m+1)}=\frac{1}{\zeta_X(m+1)} \cdot \prod_{i=1}^r \frac{q^{-\lambda_i(m+1)}}{1-q^{-\lambda_i(m+1)}}.$$

Note that there are only finitely many such $\{ P_1, \ldots, P_r\}$, as their degree is bounded by $\ell$.  Since our density definition of probability in Section 2 is finitely additive, the probabilities of being singular at each such set add to give the total probability in Theorem \ref{conjecture}: the event of a hypersurface section being singular in precisely the points of one set is certainly disjoint from the event given by a different set of points.  Meanwhile, the series contributions of each $\{ P_1, \ldots, P_r\}$ add up to all of $\zeta_X^{[\ell]}(m+1)$.  As the series terms and the probabilities were individually comparable, we're done.
\end{proof}

\bibliographystyle{alpha}

\bibliography{paper}

\end{document}